\documentclass[a4paper,11pt]{article}
\usepackage{enumerate}
\usepackage{amsmath,amsthm,amssymb}

\parskip \bigskipamount

\title{A note on large rainbow matchings in edge-coloured graphs}
\author{Allan Lo\thanks{School of Mathematics, University of Birmingham, B15 2TT, United Kingdom. Email: s.a.lo@bham.ac.uk. Allan Lo was supported by the ERC, grant no.~258345.} \and Ta Sheng Tan\thanks{Department of Pure Mathematics and Mathematical Statistics, Centre for Mathematical Sciences, University of Cambridge, Wilberforce Road, Cambridge CB3 0WB, United Kingdom. Email: T.S.Tan@dpmms.cam.ac.uk.} }
\newtheorem{thm}{Theorem}[section]

\newtheorem*{question}{Question}

\theoremstyle{remark}

\begin{document}

\maketitle
\begin{center}
 \emph{Submitted to Graphs and Combinatorics on April 19, 2012.}
\end{center}

\begin{abstract}
 A rainbow subgraph in an edge-coloured graph is a subgraph such that its edges have distinct colours. The minimum colour degree of a graph is the smallest number of distinct colours on the edges incident with a vertex over all vertices.
Kostochka, Pfender, and Yancey showed that every edge-coloured graph on $n$ vertices with minimum colour degree at least $k$ contains a rainbow matching of size at least $k$, provided $n\geq \frac{17}{4}k^2$. In this paper, we show that $n\geq 4k-4$ is sufficient for $k \ge 4$.
\end{abstract}

\section{Introduction}

Let $G$ be a simple graph, that is, no loops or multiple edges. We write $V(G)$ for the vertex set of $G$ and $\delta(G)$ for the minimum degree of $G$.  An \emph{edge-coloured graph} is a graph in which each edge is assigned a colour.  We say such an edge-coloured $G$ is \emph{proper} if no two adjacent edges have the same colour. A subgraph $H$ of $G$ is \emph{rainbow} if all its edges have distinct colours. Rainbow subgraphs are also called totally multicoloured, polychromatic, or heterochromatic subgraphs.

For a vertex $v$ of an edge-coloured graph $G$, the \emph{colour degree} of $v$ is the number of distinct colours on the edges incident with $v$. The smallest colour degree of all vertices in $G$ is the \emph{minimum colour degree of $G$} and is denoted by $\delta^c(G)$. Note that a properly edge-coloured graph $G$ with $\delta(G) \geq k$ has $\delta^c(G) \geq k$. 

In this paper, we are interested in rainbow matchings in edge-coloured graphs. The study of rainbow matchings began with a conjecture of Ryser~\cite{ryser}, which states that every Latin square of odd order contains a Latin transversal. Equivalently, for $n$ odd, every properly $n$-edge-colouring of $K_{n,n}$, the complete bipartite graph with $n$ vertices on each part, contains a rainbow copy of perfect matching. In a more general setting, given a graph $H$, we wish to know if an edge-coloured graph $G$ contains a rainbow copy of $H$. A survey on rainbow matchings and other rainbow subgraphs in edge-coloured subgraph can be found in \cite{kano}. From now onwards, we often refer to $G$ for an edge-coloured graph $G$ (not necessarily proper) of order $n$.

Li and Wang~\cite{wang2} showed that if $\delta^c(G)=k$, then $G$ contains a rainbow matching of size $\big\lceil \frac{5k-3}{12} \big\rceil$. They further conjectured that if $k\geq 4$, then $G$ contains a rainbow matching of size $\big\lceil \frac{k}{2} \big\rceil$. This bound is tight for properly edge-coloured complete graphs. LeSaulnier et al.~\cite{lesaul} proved that if $\delta^c(G) = k$, then $G$ contains a rainbow matching of size $\big\lfloor \frac{k}{2} \big\rfloor$. Furthermore, if $G$ is properly edge-coloured with $G\neq K_4$ or $|V(G)|\neq \delta(G) + 2$, then there is a rainbow matching of size $\big\lceil \frac{k}{2} \big\rceil$. The conjecture was later proved in full by Kostochka and Yancey~\cite{kostochka2}.

What happens if we have a larger graph? Wang~\cite{wang1} proved that every properly edge-coloured graph $G$ with $\delta(G) = k$ and $|V(G)|\geq \frac{8k}{5}$ contains a rainbow matching of size at least $\big\lfloor \frac{3k}{5} \big\rfloor$. He then asked if there is a function, $f(k)$, such that every properly edge-coloured graph $G$ with $\delta(G)\geq k$ and $|V(G)|\geq f(k)$ contains a rainbow matching of size $k$. The bound on the size of rainbow matching is sharp, as shown for example by any $k$-edge-coloured $k$-regular graph.
If $f(k)$ exists, then we trivially have $f(k)\geq 2k$. In fact, $f(k) > 2k$ for even $k$ as there exists $k \times k$ Latin square without any Latin transversal (see \cite{brualdi,wanless}). 
Diemunsch et at.~\cite{diemunsch1} gave an affirmative answer to Wang's question and showed that $f(k)\leq \frac{13}{5}k$. The bound was then improved to $f(k)\leq \frac{9}{2}k$ in~\cite{lo}, and shortly thereafter, to $f(k)\leq \frac{98}{23}k$ in~\cite{diemunsch2}. 

Kostochka, Pfender and Yancey~\cite{kostochka1} considered a similar problem with $\delta^c(G)$ instead of properly edge-coloured graphs. They showed that if $G$ is such that $\delta^c(G)\geq k$ and $n>\frac{17}{4}k^2$, then $G$ contains a rainbow matching of size $k$.
Kostochka~\cite{kostochka} then asked: can $n$ be improved to a linear bound in $k$? 
In this paper, we show that $n\geq 4k-4$ is sufficient for $k \ge 4$. Furthermore, this implies that $f(k)\leq 4k-4$ for $ k \ge 4$.

\begin{thm} \label{rainbowmatching}
If $G$ is an edge-coloured graph on $n$ vertices with $\delta^c(G) \geq k$, then $G$ contains a rainbow matching of size $k$, provided $n\geq 4k-4$ for $k\geq 4$ and $n\geq 4k-3$ for $k\leq 3$. 
\end{thm}

\section{Main Result}

We write $[k]$ for $\{1,2,\ldots,k\}$.
For an edge $uv$ in $G$, we denote by $c(uv)$ the colour of $uv$ and let the set of colours be $\mathbb{N}$, the set of natural numbers.

The idea of the proof is as follows. By induction, $G$ contains a rainbow matching $M$ of size $k-1$.
Suppose that $G$ does not contain a rainbow matching of size $k$.
We are going to show that there exists another rainbow matching $M'$ of size $k-1$ in $V(G) \setminus V(M)$.
Clearly, the colours of $M$ equal to the colours of $M'$. If $n \geq 4k-3$, then there exists a vertex $z$ not in $M\cup M'$.
Since $\delta^c(G) \ge k$, $z$ has a neighbour $w$ such that $zw$ does not use any colour of $M$. 
Hence, it is easy to deduce that $G$ contains a rainbow matching of size $k$.

\begin{proof}[Proof of Theorem~\ref{rainbowmatching}]
We proceed by induction on $k$. The theorem is trivially true for $k=1$. 
So fix $k>1$ and assume that the theorem is true for $k-1$. Let $G$ be an edge-coloured graph with $\delta^c(G) \geq k$ and $n=|V(G)|\geq 4k-4$ if $k\geq 4$ and $n\geq 4k - 3$ otherwise. 
Suppose that the theorem is false and so $G$ does not contain a rainbow matching of size $k$.

By induction, there exists a rainbow matching $M = \{ x_iy_i : i \in [k-1]\} $ in $G$, say with $c(x_iy_i) = i$ for each $ i \in [k-1]$.
Let $M'$ be another rainbow matching of size $s$ (which could be empty) in $G$ vertex-disjoint from $M$.
Clearly $s \le k-1$ and the colours on $M'$ is a subset of $[k-1]$, as otherwise $G$ contains a rainbow matching of size $k$.
Without loss of generality, we may assume that $M' = \{ z_iw_i : i \in [s] \}$ with $c(z_iw_i) = i$ for each $i \in [s]$.
We further assume that $M$ and $M'$ are chosen such that $s$ is maximal.
Now, let $W = V(G)  \setminus V(M \cup M')$ and $S = \{x_i, y_i, z_i, w_i : i \in [s]\}$. Clearly, if there is an edge in $W$, it must have colour in $[s]$, otherwise $G$ contains a rainbow matching of size $k$, or $s$ is not maximal.

\textbf{Fact A}
If $uw$ is an edge in $W$, then $c(uw) \in [s]$.

Furthermore, if $uv$ is an edge with $u \in S$ and $v \in W$, then $c(uv) \in [k-1]$, otherwise $G$ contains a rainbow matching of size $k$.
First, we are going to show that $s=k-1$. Suppose the contrary, $s<k-1$.
We then claim the following.

\textbf{Claim}
By relabeling the indices of $i$ (in the interval $\{s+1,s+2,\ldots,k-1\}$) and swapping the roles of $x_i$ and $y_i$ if necessary, there exist distinct vertices $z_{k-1}$, $z_{k-2}$, \dots, $z_{s+1}$ in $W$ such that for $s+1 \le i \le k-1$ the following holds for $s+1 \le i \le k-1$:
\begin{enumerate}[(a)]
 \item $y_iz_i$ is an edge and $c(y_iz_i) \notin [i]$.
 \item Let $T_i$ be the vertex set $\{x_j, y_j, z_j : i \le j \le k-1\}$.
       For any colour $j$, there exists a rainbow matching of size $k-i$ on $T_i$ which does not use any colour in $[i-1] \cup \{j\}$.
 \item Let $W_i = W \backslash \{z_i, z_{i+1}, \dots, z_{k-1} \}$.
       If $x_i w$ is an edge with $w \in W_i$, then $ c(x_i w ) \in [s]$.
 \item If $uw$ is an edge with $u \in S$ and $w \in W_i$, then $c(uw) \in [i-1]$.
 \item If $uw$ is an edge with $u \in S \cup T_i \cup W$ and $w \in W_i$, then $c(uw) \in [i-1]$ or $u \in \{y_i, \dots, y_{k-1} \}$.
\end{enumerate} 

\textit{Proof of Claim.} 
Let $W_k=W$ and $T_k=\emptyset$. Observe that part (d) and (e) of the claim hold for $i = k$. For each $i = k-1, k-2, \dots, s+1$ in terms, we are going to find $z_i$ satisfying (a) -- (e). 
Suppose that we have already found $z_{k-1}, z_{k-1}, \ldots, z_{i+1}$.

%Part (a)
Note that $|W_{i+1}| \ge n - 2(k-1) - 2s - (k-i-1) \ge 1$, so $W_{i+1} \ne \emptyset$.
Let $z$ be a vertex in $W_{i+1}$. 
By the colour degree condition, $z$ must incident with at least $k$ edges of distinct colours, and in particular, at least $k-i$ distinct coloured edges not using colours in $[i]$. 
Then, there exists a vertex $u\in \{x_j,y_j: s+1\leq j \leq i\}$ such that $uz$ is an edge with $c(uz)\notin [i]$ by part (e) of the claim for the case $i+1$.
Without loss of generality, $u=y_i$ and we set $ z_i=z $.

%Part (b)
Part (b) of the claim is true for colour $j\neq i$, simply by taking the edge $x_iy_i$ together with a rainbow matching of size $k-i-1$ on $T_{i+1}$ which does not use any colour in $[i] \cup \{j\}$. For colour $j=i$, we take the edge $y_iz_i$ together with a rainbow matching of size $k-i-1$ on $T_{i+1}$ which does not use any colour in $[i] \cup \{c(y_iz_i)\}$.

%Part (c)
To show part (c) of the claim, let $x_i w$ be an edge for some $w \in W_i$. By part (b) of the claim for the case $i+1$, there exists a rainbow matching $M''$ of size $k-i-1$ on $T_{i+1}$ which does not use any colour in $[i]\cup \{c(y_iz_i)\}$. Set ${M}_0 = \{x_j y_j : j \in [i-1] \} \cup M'' \cup \{y_i z_i\}$. Then, $M_0$ is a rainbow matching of size $k-1$ vertex-disjoint from $M'$. 
Now, by considering the pair $(M_0,M')$ instead of $(M,M')$, we must have $c(x_i w) \in [s]$. 
Otherwise, $G$ contains a rainbow matching of size $k$ or $s$ is not maximal.

%Part (d)
Let $uw$ be an edge with $u \in S$, $w \in W_i$ and $ c(u w ) \notin [i-1]$.
Pick a rainbow matching $M_u$ of size $s$ on $S \setminus \{u\}$ with colours $[s]$, and a rainbow matching $M_u'$ of size $k-i$ on $T_i$ which does not contain any colour in $[i-1]\cup \{c(uw)\}$.
Then, $\{uw\} \cup M_u \cup M_u' \cup \{x_j y_j : s+1\leq j \leq i-1 \}$ is a rainbow matching of size $k$ in $G$, a contradiction.
So $c(u w)\in [i-1]$ for any $u\in S$ and $w\in W$, showing part (d) of the claim.

%Part (e)
Part (e) of the claim follows easily from Fact~A, (c) and (d). This completes the proof of the claim. \qed

Recall that $s<k-1$. So we have $|W_{s+1}| = n - 2(k-1) - 2s - (k-1-s) \geq k-1-s \geq 1$. Pick a vertex $w\in W_{s+1}$. By part (e) of the claim, $w$ adjacent to vertices in $\{y_{s+1},y_{s+2},\ldots,y_{k-1}\}$ or $w$ incident with edges of colours in $[s]$. Hence, $w$ has colour degree at most $k-1$, which contradicts $\delta^c(G) \geq k$. Thus, we must have $s=k-1$ as claimed.
In summary, we have $M = \{ x_iy_i : i \in [k-1]\}$ and $M' = \{ z_iw_i : i \in [k-1] \}$ with $c(x_iy_i) = i = c(z_iw_i) $ for $i\in [k-1]$.

Now, if $n\geq 4k-3$, then $  V(G) \neq V(M \cup M')$. 
Pick a vertex $w\notin V(M \cup  M')$ and since $w$ has colour degree at least $k$, there exists a vertex $u$ such that $uw$ is an edge and $c(uw)\notin [k-1]$. 
It is easy to see that $G$ contains a rainbow matching of size $k$, contradicting our assumption. 
Therefore, we may assume $n=4k-4$ and $k\geq 4$. 

Since $\delta^c(G)\geq k$, any vertex $u\in \{x_1,y_1,z_1,w_1\}$ must have a neighbour $v$ such that $c(uv)\notin [k-1]$. If $v\notin \{x_1,y_1,z_1,w_1\}$, then $G$ contains a rainbow matching of size $k$. 
So, without loss of generality, $x_1 z_1$ and $y_1 w_1$ are edges in $G$ with $c(x_1 z_1),c(y_1 w_1)\notin [k-1]$.
By symmetry, we may assume that for each $i\in [k-1]$, $x_i z_i$ and $y_i w_i$ are edges in $G$ with $c(x_i z_i), c(y_i w_i) \notin [k-1]$.
As $\delta^c(G)\geq k\geq 4$, $x_1$ must have a neighbour $v\notin \{y_1,z_1,w_1\}$ with $c(x_1 v)\neq 1$. Without loss of generality, we may assume $v=z_j$ for some $j$ and $c(x_1 z_j) = 2$. Now, $\{x_1z_j, z_1 w_1, y_2 w_2, \} \cup \{x_i y_i : i\in \{3,4,\ldots,k-1\}$ is a rainbow matching of size $k$ in $G$, which again is a contradiction. This completes the proof of the theorem. 
\end{proof}

\section{Remarks}

In Theorem~\ref{rainbowmatching}, the bound on $n$, the number of vertices, is sharp for $k=2,3$ (and trivially for $k=1$), as shown by properly 3-edge-coloured $K_4$ for $k=2$ and by properly 3-edge-coloured two disjoint copies of $K_4$ for $k=3$.
However, we do not know if the bound is sharp for $k\geq 4$.

\begin{question}
Given $k$, what is the minimum $n$ such that every edge-coloured graph $G$ of order $n$ with $\delta^c(G) = k$ contains a rainbow matching of size $k$? 
\end{question}

\subsection*{Acknowledgment}
The authors thank Alexandr Kostochka for suggesting the problem during `Probabilistic Methods in Graph Theory' workshop at University of Birmingham, United Kingdom.
We would also like to thank Daniela K\"uhn, Richard Mycroft and Deryk Osthus for organizing this nice event.


\begin{thebibliography}{99}

\bibitem{brualdi} R. A. Brualdi and H. J. Ryser, \emph{Combinatorics matrix theory}, Encyclopedia of Mathematics and its Applications, vol. 39, Cambridge University Press, 1991.
\bibitem{diemunsch1} J. Diemunsch, M. Ferrara, C. Moffatt, F. Pfender, and P. S. Wenger, \emph{Rainbow matching of size $\delta(G)$ in properly-colored graphs}, arXiv:1108.2521.
\bibitem{diemunsch2} J. Diemunsch, M. Ferrara, A. Lo, C. Moffatt, F. Pfender, and P. S. Wenger, \emph{Rainbow matching of size $\delta(G)$ in properly-colored graphs}, submitted.
\bibitem{kano} M. Kano and X. Li, \emph{Monochromatic and heterochromatic subgraphs in edge-colored graphs -- a survey}, Graphs Combin. \textbf{24} (2008), 237-263.
\bibitem{kostochka} A. Kostochka, \emph{personal communication}.
\bibitem{kostochka1} A. Kostochka and M. Yancey, \emph{Large rainbow matchings in edge-coloured graphs}, Combinatorics, Probability and Computing \textbf{21} (2012), 255-263.
\bibitem{kostochka2} A. Kostochka , F. Pfender and M. Yancey, \emph{Large rainbow matchings in large graphs}, arXiv:1204.3193.
\bibitem{lesaul} T. D. LeSaulnier, C. Stocker, P. S.Wenger, and D. B. West, \emph{Rainbow matching in edge-colored graphs}, Electron. J. Combin. \textbf{17} (2010), \#N26.
\bibitem{wang2} H. Li and G. Wang, \emph{Heterochromatic matchings in edge-colored graphs}, Electron. J. Combin. \textbf{15} (2008), \#R138.
\bibitem{lo} A. Lo, \emph{A note on rainbow matchings in properly edge-coloured graphs}, arXiv:1108.5273.
\bibitem{ryser} H. J. Ryser, \emph{Neuere probleme der kombinatorik}, Vortrage \"uber Kombinatorik Oberwolfach, Mathematisches Forschungsinstitut Oberwolfach (1967), 24-29.
\bibitem{wang1} G. Wang, \emph{Rainbow matchings in properly edge colored graphs}, Electron. J. Combin. \textbf{18} (2011), \#P162.
\bibitem{wanless} I. M. Wanless, \emph{Transversals in latin squares: a survey}, Surveys in Combinatorics 2011, London Math. Soc., 2011.



\end{thebibliography}
\end{document}